\documentclass[12pt]{amsart}
\usepackage{amsmath}
\usepackage{amssymb,latexsym,amsthm,amsfonts,amscd,pgf,enumerate,ragged2e}
\usepackage{tikz,amsmath}
\usepackage{subcaption}
\usepackage{graphicx}
\usepackage{float}
\usepackage{ragged2e}

\usepackage[pdftex]{hyperref}
\newtheorem{theorem}{Theorem}[section]
\newtheorem{lemma}[theorem]{Lemma}
\newtheorem{proposition}[theorem]{Proposition}
\newtheorem{corollary}[theorem]{Corollary}
\newtheorem{remark}[theorem]{Remark}

\numberwithin{equation}{section}

\newcommand{\Nn}{\mbox{$\mathbb N$}}

\begin{document}

\baselineskip=15.5pt

\title[Graded Betti numbers of some circulant graphs]{Graded Betti numbers of some circulant graphs}

\author[S. Anand]{Sonica Anand}

\address{Mehr Chand Mahajan DAV College for Women, Sector-36A, Chandigarh - 160 036 , India}

\email{sonica.anand@gmail.com}

\author[A. Roy]{Amit Roy}

\address{IISER Mohali,
Knowledge City, Sector 81, SAS Nagar, Punjab -140 306, India.}

\email{amitroy@iisermohali.ac.in, amitiisermohali493@gmail.com}
\subjclass[2010]{13F55, 13H10, 05C75, 05E45}

\date{}

\begin{abstract}
Let $G$ be the circulant graph $C_n(S)$ with $S \subseteq \{1, 2, \dots, \lfloor \frac{n}{2} \rfloor\}$, and let $I(G)$ denote the edge ideal in the polynomial ring $R=\mathbb{K}[x_0, x_1, \dots, x_{n-1}]$ over a field $\mathbb{K}$. In this paper, we compute the $\Nn$-graded Betti numbers of the edge ideals of three families of circulant graphs $C_n(1,2,\dots,\widehat{j},\dots,\lfloor \frac{n}{2} \rfloor)$, $C_{lm}(1,2,\dots,\widehat{2l},\dots, \widehat{3l},\dots,\lfloor \frac{lm}{2} \rfloor)$ and $C_{lm}(1,2,\dots,\widehat{l},\dots,\widehat{2l},\dots, \widehat{3l},\dots,\lfloor \frac{lm}{2} \rfloor)$. Other algebraic and combinatorial properties like regularity, projective dimension, induced matching number and when such graphs are well-covered, Cohen-Macaulay, Sequentially Cohen-Macaulay, Buchsbaum and $S_2$ are also discussed.
\vspace{0.2cm}

\noindent
{\sc Key words}: Circulant graphs, Edge ideals, Betti numbers, Castelnuovo-Mumford regularity, Cohen-Macaulay, Buchsbaum, well-covered.

\end{abstract}
\maketitle


\maketitle

\section{Introduction}\label{introduction}

Let $G=(V(G), E(G))$ denote a finite simple graph with vertex set $V(G)=\{0, 1, \dots, n-1\}$ and edge set $E(G)$. Let $R=\mathbb K[x_0, x_1, \dots, x_{n-1}]$ be a polynomial ring in $n$ variables over a field $\mathbb{K}$. We can associate to $G$ the quadratic square-free monomial ideal 
\[ I(G)=\langle x_ix_j|\{i,j\} \in E(G) \rangle \subseteq R, \]
called the {\em edge ideal} of $G$. Edge ideals were first introduced by Villarreal \cite{RV}. They are mainly studied to investigate relations between algebraic properties of edge ideals and combinatorial properties of the corresponding graphs. We mainly focus on describing invariants of $I(G)$ in terms of $G$.\\
For the edge ideal $I(G)$ in $R=\mathbb{K}[x_0, x_1, \dots, x_{n-1}]$ there exists an $\Nn$-graded minimal free resolution 
\[ \mathcal{F}: 0 \rightarrow F_p \rightarrow F_{p-1} \rightarrow \dots \rightarrow F_0 \rightarrow R/I(G) \rightarrow 0,  \]
where $p\leq n$, $F_i=\oplus_jR(-j)^{\beta_{i,j}}$ and $R(-j)$ is the graded free $R$-module obtained by shifting the graded components of $R$ by $j$. The numbers $\beta_{i,j}$ are called the ith graded Betti numbers of $R/I(G)$ in degree $j$ and we write $\beta_{i,j}(G)$ for $\beta_{i,j}(R/I(G))$. The length $p$ of the resolution is called the {\em projective dimension} of $R/I(G)$ and is denoted by $\mathrm{pd}(R/I(G))$ (we write $\mathrm{pd}(G)$ for $\mathrm{pd}(R/I(G))$), i.e.,
\begin{center}
$\mathrm{pd}(G)$= max $\{i | \beta_{i,j}(G)\neq 0$ for some $j \}.$
 \end{center}
  Betti numbers and projective dimension are among the most important invariants in a graded minimal free resolution. Moreover, the Castelnuovo-Mumford regularity (or simply the regularity) of $R/I(G)$ is another important invariant encoded in the minimal free resolution of $R/I(G)$. 
  \newpage
  Denoted by $\mathrm{reg}(R/I(G))$ (or simply $\mathrm{reg}(G)$), the {\em regularity} of $R/I(G)$ is defined as
 \begin{center}
reg(G) = max $\{ j-i| \beta_{i,j}(G) \neq 0 \}.$ 
\end{center}

Given an integer $n\geq 1$ and a subset $S \subseteq\{1,2,\dots,\lfloor \frac{n}{2} \rfloor\}$, the {\em circulant graph} $C_n(S)$ is a simple graph with $V(G)=\{0,1,\dots,n-1\}$ and $E(G)=\{\{i,j\}||i-j|_n \in S\}$, where $|k|_n= \min \{|k|, n-|k|\}$. By abuse of notation we write $C_n(a_1,a_2,\dots,a_s)$ instead of $C_n(\{a_1,a_2,\dots,a_s\})$. Circulant graphs belong to the family of Cayley graphs and are considered as a generalization of the cycles as $C_n=C_n(1)$. The complete graph $K_n=C_n(1,2,\dots,\lfloor \frac{n}{2} \rfloor)$ is also a circulant graph. Various combinatorial and algebraic aspects of circulant graphs have been studied in \cite{BH}, \cite{EMT}, \cite{RH}, \cite{MM}, \cite{AM1}, \cite{FR}, \cite{VMVTW}, \cite{VTUP}. Circulant graphs have appeared in the literature in a number of applications such as networks \cite{BIP}, connectivity \cite{BT}, error-correcting codes \cite{ST} and even music \cite{BH0} becuase of their regular structure.

Recently, Uribe-Packza and Van Tuyl \cite{VTUP} determined the regularity of the circulant graph $H_1=C_n(1,\dots,\widehat{j},\dots,\lfloor \frac{n}{2} \rfloor)$. In this paper we compute the $\Nn$-graded Betti numbers of $H_1$ by showing that $H_1$ is isomorphic to the graph $\underbrace{C_k^c*\dots*C_k^c}_{d\text{-times}}$, with $d=\gcd(n,j)$ and $k=\frac{n}{d}$. We provide an alternate method to compute regularity of $H_1$. 

In order to provide the generalization to the work of  Mousivand and Makvand \cite{MM}, we consider the circulant graph $H_2=C_{lm}(1,2,\dots,\widehat{2l},\dots, \widehat{3l},\dots,\lfloor \frac{lm}{2} \rfloor)$, where $l,m\geq 2$. For $l=2$, Mousivand and Makvand computed the Betti numbers by showing that it is isomorphic to the join $C_m*C_m$. In this article we show that for every $l\geq 2$, the graph $H_2$ is isomorphic to the join $\underbrace{C_m*\dots *C_m}_{l\text{-times}}$. Consequently, formulas for all the $\Nn$-graded Betti numbers of $H_2$ are also obtained. Finally, we show that a complete multipartite graph $K_{m_1,\dots ,m_l}$ is a circulant graph if and only if $m_i=m$ for each $1\leq i\leq l$. In this case, the graph is isomorphic to the circulant graph $H_3=C_{lm}(1,\dots ,\widehat{l},\dots ,\widehat{2l},\dots ,\lfloor\frac{lm}{2}\rfloor)$. Formulas for the graded Betti numbers of $H_3$ are also obtained. For each of these families of graphs we compute the regularity, projective dimension and induced matching number. Properties such as when these graphs are well-covered, Buchsbaum, Cohen-Macaulay, sequentially Cohen-Macaulay, shellable, vertex decomposable or $S_2$ are also discussed.

Our paper is structured as follows. In Section {\ref{Pre}}, we recall some basic definitions and known results from graph theory and commutative algebra that we will use throughout this paper. Moreover, we state and prove two results (Proposition \ref{joinproperties} and \ref{joinmultipartite}) regarding various algebraic and combinatorial properties of join of graphs. In Section \ref{jhat}, we provide formulas for computing graded Betti numbers for the family of circulant graphs $C_n(1,2,\dots,\widehat{j},\dots,\lfloor \frac{n}{2} \rfloor)$ (see Theorem \ref{Bettijoindtimes}). In section {\ref{power}}, we study the join of cycles and deduce formulas for the graded Betti numbers of the circulant graph $C_{lm}(1,2,\dots,\widehat{2l},\dots, \widehat{3l},\dots,\lfloor \frac{lm}{2} \rfloor)$ (see Theorem \ref{joinBetti1}, \ref{joinBetti2} and \ref{regularityBettinumber}). Finally, in Section {\ref{multi}}, we deduce  a formula for computing graded Betti numbers of the circulant graphs of the form $C_{lm}(1,2,\dots,\widehat{l},\dots,\widehat{2l},\dots, \widehat{3l},\dots,\lfloor \frac{lm}{2} \rfloor)$ by proving that such a graph is isomorphic to complete multipartite graph $K_{\underbrace{m,m,\dots,m}_{l-times}}$ (see Lemma \ref{multipartite} and Theorem \ref{Bettimulti}). In each of these sections, properties like projective dimension, well-covered, Cohen-Macaulayness, etc. for the corresponding families of graphs are discussed (see Theorem \ref{hatproperites}, \ref{cyclecombinatorialproperties} and \ref{multipartiteproperties}).

\section{\bf{Preliminaries}}\label{Pre}

In this section, we recall some basic definitions and some known results that will be used throughout this paper. In addition, we prove some basic results that are required in the subsequent sections.


\subsection {Graph Theoretic and Algebraic Preliminaries}

Let $G=(V(G),E(G))$ be a finite simple graph.  The {\em neighbourhood} of $x \in V(G)$ is the set $N_G(x)=\{y | \{x,y\} \in E(G) \}.$ The closed neighbourhood of $x$ is $N_G[x]=N_G(x) \cup \{x\}$. The degree of $x$ is $deg(x)=|N_G(x)|$. A graph $H=(V(H),E(H))$ is a {\em subgraph} of $G$ if $V(H) \subseteq V(G)$ and $E(H) \subseteq E(G)$. For $W\subseteq V(G)$, the {\em induced subgraph} of $G$ on the vertex set $W$, denoted by $G_W$, is the graph whose edge set consists of all the edges in $G$ that have both endpoints in $W$ (i.e. $E(G_W)=\{\{x,y\} \in E(G)|\{x,y\}\subseteq W\}$). The {\em complement} of a graph $G$ denoted by $G^c$ is the graph $(V(G^c),E(G^c))$ where $V(G^c)=V(G)$ and $E(G^c)=\{ \{x,y\} | \{x,y\} \notin E(G)\}$.

Let $G$ and $H$ be two simple graphs with disjoint set of vertices, i.e., $V(G)\cap V(H)=\emptyset$. The {\em join} of $G$ and $H$, denoted by $G*H$, is the graph on the vertex set $V(G)\cup V(H)$ with edge set given by $E(G*H)=E(G)\cup E(H)\cup\{\{x,y\}:x\in V(G)\text{ and }y\in V(H)\}$.

Given a graph $G$, an {\em induced matching} of $G$ is an induced subgraph consisting of pairwise disjoint edges. The maximum number of edges in an induced matching is called the {\em induced matching number} of $G$ and is denoted by $\nu(G)$.

We recall some basic facts about simplicial complexes. A {\em simplicial complex} on a vertex set $V=\{x_1, x_2, \dots, x_n\}$ is a set $\Delta$ whose elements are subset of $V$ such that (a) if $F \in \Delta$ and $G \subset F$, then $G \in \Delta$, and (b) for each $i=1, 2, \dots, n$, $\{x_i\} \in \Delta$. Note that the set $\emptyset \in \Delta$.
An element $F \in \Delta$ is called a {\em face}. The maximal elements of $\Delta$, with respect to inclusion, are called the {\em facets} of $\Delta$. If $\{F_1, F_2, \dots, F_t \}$ is a complete list of the facets of $\Delta$, we will sometimes write $\Delta = \langle F_1, \dots, F_t \rangle$. The {\em dimension} of a face $F \in \Delta$, denoted by $\dim F$, is given by $\dim F=|F|-1$, where we make the convention that $\dim \emptyset = -1$. The {\em dimension} of $\Delta$, denoted by $\dim \Delta$, is defined to be $\dim \Delta = max_{F \in \Delta} \{ \dim F \}$. A simplicial complex is called {\em pure} if all its facets have the same dimension. A simplicial complex having exactly one facet is called a {\em simplex}.

Let $\Delta$ and $\Delta'$ be two simplicial complexes with vertex sets $V$ and $V'$ respectively. The union $\Delta \cup \Delta'$ is a simplicial complex with vertex set $V \cup V'$ and $F$ is a face of $\Delta \cup \Delta'$ if and only if $F$ is a face of $\Delta$ or $\Delta'$. 
We say that a subset $W \subset V(G)$ is a {\em vertex cover} of $G$ if $e \cap W \neq \emptyset$  for all edges $e \in E(G)$. The complement of a vertex cover is an independent set. A graph G is called {\em well-covered} if every minimal vertex cover (with respect to the partial order of inclusion) has the same cardinality. Via the duality between vertex covers and independent sets, being well-covered is equivalent to the property that every maximal independent set has the same cardinality.
The cardinality of the largest independent set in $G$ is denoted by $\alpha(G)$. The family of all independent sets of $G$ is a simplicial complex on the vertex set $V(G)$, called as {\em independence complex} of $G$ and is denoted by $\Delta_G$,

\[ \Delta_G= \{W \subset V(G) | \text{W is an independent set of G} \}. \] 
Note that, $\dim\Delta_G=\alpha(G)-1$. A graph is well-covered if and only if $\Delta_G$ is a pure simplicial complex.

For a simplicial complex $\Delta$, if $F \in \Delta$ is a face, then the link of $F$ is the simplicial complex 

\begin{center}
$ \mathrm{lk}_{\Delta}(F)=\{H \in \Delta | H \cap F =\emptyset$ and $H \cup F \in \Delta \}$,
\end{center}
and deletion of $F$ is the simplicial complex 

\begin{center}
 $\mathrm{del}_{\Delta}(F)=\{ H \in \Delta | H \cap F =\emptyset \} $.
\end{center}

When $F=\{x_i\}$, then we simply write $\mathrm{lk}_{\Delta}(x_i)$ or $\mathrm{del}_{\Delta}(x_i)$.

Given a pure simplicial complex $\Delta$, we say $\Delta$ is {\em vertex decomposable} if either $\Delta$ is a simplex, or there exists a vertex $x$ such that both $\mathrm{lk}_{\Delta}(x)$ and $\mathrm{del}_{\Delta}(x)$ are vertex decomposable (see \cite{PB}). A simplicial complex $\Delta$ is called {\em shellable} if $\Delta$ is pure and there exists an ordering of facets $F_1<F_2<\dots <F_r$ such that for all $1\leq j<i\leq r$, there is some $x\in F_i\setminus F_j$ and some $k\in\{1,\dots ,i-1\}$ for which $F_i\setminus F_k=\{x\}$.

A graph $G$ is called a {\em Cohen-Macaulay} graph if the independence complex $\Delta_G$ satisfies the following: $\widetilde{H}_i(\mathrm{lk}_{\Delta_G}(F);\mathbb{K})=0$ for all $F\in\Delta_G$ and for all $i < \dim \mathrm{lk}_{\Delta_G} (F)$ (here $\widetilde{H}_i(-;\mathbb{K})$ is the ith reduced simplicial homology group). 
$G$ is Buchsbaum if $G\setminus N_G[x]$ is Cohen-Macaulay for every vertex $x$ in $V(G)$. Note that, the definitions depend on the characteristic of the field $\mathbb{K}$. The graph $G$ is said to satisfy {\em Serre's condition} $S_2$ if $\mathrm{lk}_{\Delta_G}(F)$ is connected for every face $F$ of $\Delta_G$ having $\dim \mathrm{lk}_{\Delta_G}(F)\geq 1$ (see \cite{T}).

Let $\Delta$ be a simplicial complex. The {\em pure $i^{th}$ skeleton} of $\Delta$ is the subcomplex $\Delta^{[i]}$ of $\Delta$ whose facets are the faces $F$ of $\Delta$ with $\dim F=i$. A simplicial complex $\Delta$ is said to be {\em sequentially Cohen-Macaulay} if $\Delta^{[i]}$ is Cohen-Macaulay for all $i$. A graph $G$ is called a sequentially Cohen-Macaulay graph if $\Delta_G$ is a sequentially Cohen-Macaulay simplicial complex. 

Given a simplicial complex $\Delta$ on the vertex set $V=\{x_1,\dots ,x_n\}$, we can associate it with a monomial ideal $I_{\Delta}$ in the polynomial ring $R=\mathbb{K}[x_1, x_2, \dots, x_n]$ (for a field $\mathbb{K}$) in the following way. For every subset $F$ of $V$, we define a monomial $x_F:=\prod_{x_i \in F} x_i$ in $R$. Then the ideal $I_{\Delta}:=\langle x_F: F \notin \Delta \rangle $ is called the {\em Stanley-Reisner} ideal of $\Delta$ and the quotient ring $\mathbb{K}[\Delta]= R/I_{\Delta}$ is called the {\em Stanley-Reisner} ring.

A simplicial complex $\Delta$ is called {\em Cohen-Macaulay (or Buchsbaum, $S_2, \dots $)} over a field $\mathbb{K}$ if its Stanley-Riesner ring $\mathbb{K}[\Delta]$ is  Cohen-Macaulay (or Buchsbaum, $S_2, \dots $). Given a finite simple graph $G$, the edge ideal $I(G)$ is the Stanley-Reisner ideal associated to the independence complex $\Delta_G$ of $G$. The properties that $G$ is Cohen-Macaulay (or Buchsbaum, $S_2,\dots$) are equivalent to the fact that the Stanley-Reisner ring $\mathbb{K}[\Delta_G]$ has those properties. 

There is a strong connection between the topology of the simplicial complex $\Delta$ and the structure of the free resolution of $\mathbb{K}[\Delta]$. Let $\beta_{i,j}^{\mathbb{K}}(\Delta)$ denote the $\Nn$ - graded Betti numbers of the Stanley-Reisner ring $\mathbb{K}[\Delta]$. Betti numbers may in general be dependent on (characteristic of) the field $\mathbb{K}$. By considering a fix field $\mathbb{K}$, we simply write $\beta_{i,j}(\cdot)$ for $\beta_{i,j}^{\mathbb{K}}(\cdot)$

One of the most well-known results to compute Betti numbers is the Hochster's formula \cite{MH}. However, this formula is somewhat daunting to use for computing all Betti numbers of circulant graphs with large number of vertices because one has to compute the dimensions of all the homology groups. In this paper, we see the circulant graphs from different lens and collectively use \cite{MH}, Hochster's formula and \cite{AM}, Corollary 3.4 to compute the Betti numbers and the regularity.
 
 
\subsection {Known results}

In this subsection we summarize some well-known results about the above mentioned properties.

\begin{theorem}\cite[Theorem 2.3]{EMT}\label{propeties results}
  Let $\Delta$ be a pure simplicial complex on $V=\{x_1,\dots ,x_n\}$.
  
  \begin{enumerate}[(i)]
  \item The following implications hold for $\Delta$:
   \[\text{vertex decomposable}\implies\text{shellable}\implies\text{Cohen-Macaulay}\implies\text{Buchsbaum.}\]
   \item If $\dim\Delta=0$, then $\Delta$ is vertex decomposable (and thus, shellable, Cohen-Macaulay, and Buchsbaum).
   \item If $\dim\Delta=1$, then $\Delta$ is vertex decomposable/shellable/Cohen-Macaulay if and only if $\Delta$ is connected. If $\Delta$ is not connected, then $\Delta$ is Buchsbaum but not Cohen-Macaulay.
   \item If $\dim\Delta\geq 2$ and $\Delta$ is Cohen-Macaulay, then $\Delta$ is connected.
   \end{enumerate}
  \end{theorem}

  The following two results are taken from the thesis of Jacques.

\begin{theorem}\cite[Theorem 5.3.8]{SJ}\label{Jac}
The $\mathbb N$-graded Betti numbers of the complete multipartite graph $K_{n_1,\dots ,n_t}$ are independent of the characteristic of field $\mathbb{K}$ and may be written as
\[\beta_{i,d}(K_{n_1,\dots ,n_t})=\begin{cases}
                                      \sum_{l=2}^t(l-1)\sum_{\underset{\alpha_1,\dots,\alpha_l\geq 1}{\underset{j_1<\dots <j_l,}{\alpha_1+\dots +\alpha_l=d,}}}\binom{n_{j_1}}{\alpha_1}\dots\binom{n_{j_l}}{\alpha_l}\hspace{1cm}\text{if}\hspace{1cm} d=i+1, \\
                                     \hspace{2cm} 0 \hspace{5.5cm}\text{if} \hspace{1cm}d\neq i+1.
                                     \end{cases}
\]
\end{theorem}

\begin{theorem}\cite[Theorem 4.2.6]{SJ} \label{pd}
If $G$ is a graph such that $G^c$ is disconnected then \[pd(R/I(G))=|V(G)|-1. \]
\end{theorem}

The following propositions are the main tools in this article for computing the Betti numbers and the regularity.
\begin{proposition}\cite[Hochster's formula]{MH}\label{Hochster}
For a simplicial complex $\Delta$ on vertex set $[n]$ and $\mathbb{K}[\Delta]=R/I_{\Delta}$ denotes its Stanley-Riesner ring. Then for $i \geq 1$, the Betti numbers $\beta_{i,d}$ of $R/I_{\Delta}$ are given by 
\[ \beta_{i,d}(\mathbb{K}[\Delta])= \sum_{\underset{|W|=d}{W \subset [n]}} \dim_{\mathbb{K}} \tilde{H}_{d-i-1}(\Delta[W]; \mathbb{K}). \] Here $\Delta[W]$ denotes the simplicial complex induced on the vertex set $W$.
\end{proposition}

\begin{proposition}\cite[Corollary 3.4]{AM}\label{joinoftwographs}
Let $G$ and $H$ be two simple graphs with disjoint vertex sets having $m$ and $n$ vertices, respectively. Then the $\Nn$-graded Betti numbers $\beta_{i,d}(G*H)$ may be expressed as
\begin{align*}
  & \sum_{j=0}^{d-2}\left\{\binom{n}{j}\beta_{i-j,d-j}(G)+\binom{m}{j}\beta_{i-j,d-j}(H)\right\} \quad\text{if}\quad d\neq i+1, \\
  & \sum_{j=0}^{d-2}\left\{\binom{n}{j}\beta_{i-j,d-j}(G)+\binom{m}{j}\beta_{i-j,d-j}(H)\right\}+\sum_{j=1}^{d-1}\binom{m}{j}\binom{n}{d-j} \quad\text{if}\quad d=i+1.
\end{align*}
\end{proposition}

%
%
%
%
%
%

\begin{proposition}\cite[Proposition 3.12]{AM}\label{joinregularity}
Let $G$ and $H$ be two simple graphs with disjoint vertex sets and one of them having atleast one edge. Then 
\[\mathrm{reg}(G*H)=\max\{\mathrm{reg}(G),\mathrm{reg}(H)\}.\]
\end{proposition}
 
 In the following proposition we record a number of properties for the cycle graphs.
 
  \begin{proposition}\label{cycleproperties}
   Let $m\geq 3$ be an integer and $C_m$ denote the cycle of length $m$. Then 
     \begin{enumerate}[(i)]
    \item $C_m$ is well-covered/Buchsbaum if and only if $m\leq 5$ or $m=7$.
   \item $C_m$ is vertex decomposable/shellable/Cohen-Macaulay/sequentially Cohen-Macaulay if and only if $m\in\{3,5\}$.
   \item $C_m$ is $S_2$ if and only if $m\in\{3,5,7\}$.
   \end{enumerate}
  \end{proposition}
  \begin{proof}
   The statements regarding when $C_m$ is well-covered/Buchsbaum/sequentially Cohen-Macaulay can be deduced from the proof of Proposition 3.2 in \cite{MM}. Statements regarding Cohen-Macaulay and $S_2$ properties are the content of \cite{RV1}, Corollary 7.3.19 and \cite{HYZN}, Proposition 1.6, respectively. The criteria for vertex decomposability and shellability can be deduced from Theorem 3.4 in \cite{VMVTW}. \qed
   \end{proof}

   \subsection{Some Basic Results} The following propositions are well known. We include them for the sake of completeness.
  
  \begin{proposition}\label{joinproperties}
   Let $d\geq 2$ be an integer. Suppose $G_1,\dots ,G_d$ are $d$ number of finite simple graphs with disjoint vertex sets. Let $G=G_1*\dots*G_d$. Then
   \begin{enumerate}[(i)]
   \item The induced matching number $\nu(G)=\begin{cases}
   \max_i\{\nu(G_i)\}\quad\text{if }\nu(G_i)\neq 0\text{ for some $i$ },\\
   1\hspace{2.5cm}\text{otherwise.}
                                             \end{cases}
$.
   
    \item $G$ is well-covered if and only if all $G_j$'s are well-covered and for each $i\neq j$
   \[ \alpha(G_i)=\alpha(G_j). \]
   \item $G$ is vertex decomposable/shellable/Cohen-Macaulay (or $S_2$) if and only if $G_j$'s are complete graphs for all $j$.
   \item $G$ is sequentially Cohen-Macaulay if and only if $G_t$ is sequentially Cohen-Macaulay for some $1\leq t\leq d$ and $G_j$'s are complete for all $j\neq t$.
   \item $G$ is Buchsbaum if and only if each $G_i$ is Buchsbaum for $1\leq i\leq d$.
   \end{enumerate}
   \end{proposition}

  \begin{proof}
    Note that $\Delta_G=\sqcup_{i=1}^d\Delta_{G_i}$, where $\Delta_G$ is the independence complex of $G$.
   \begin{enumerate}[(i)]
    \item Follows from the definition. Note that, $\nu(G_i)=0$ if and only if $G_i$ consists of isolated vertices.
    \item The statement follows from the fact that the maximal independent sets of $G_j$ for all $j$ are the maximal independent sets of $G$.
    \item First note that for $S_2$ property the statement directly follows from definition. Now we consider the Cohen-Macaulay property. If some $G_i$ is not complete then $\dim\Delta_{G_i}\geq 1$ and consequently $\dim\Delta_{G}\geq 1$. Since Cohen-Macaulay simplicial complexes of positive dimension are connected we have that $G$ is not Cohen-Macaulay. The converse follows from the fact that all $0$-dimensional simplicial complexes are Cohen-Macaulay (see Theorem \ref{propeties results}). The statement about vertex-decomposability and shellability follows from the fact that all $0$-dimensional simplicial complexes are vertex decomposable/shellable and also from the well-known hierarchy of conditions:
    \begin{align}\label{hierarchy of conditions}
    \text{vertex decomposable}\implies \text{shellable}\implies\text{Cohen-Macaulay.}
    \end{align}

    \item Let $G$ be a sequentially Cohen-Macaulay graph. Suppose $G_r$ and $G_s$ are not complete for some $r\neq s$. Then $\Delta_G^{[1]}$ is $1$-dimensional and disconnected and hence not Cohen-Macaulay, a contradiction. Therefore, we must have atmost one $G_j$ is not complete. Clearly, in that case $\Delta_G^{[l]}=\Delta_{G_j}^{[l]}$ for all $l>0$ and hence $G_j$ is sequentially Cohen-Macaulay. Converse part is clear from the fact that $\Delta_G^{[l]}=\Delta_{G_j}^{[l]}$ for all $l>0$.
   \item Let $G$ be a Buchsbaum graph. Let $x\in V(G_i)$ for some $i$. Then $G\setminus N_G[x]=G_i\setminus N_{G_i}[x]$ and hence $G_i\setminus N_{G_i}[x]$ is Cohen-Macaulay for all $x\in V(G_i)$. Consequently, $G_i$ is Buchsbaum. Conversely, take $x\in V(G)$, then $x\in V(G_i)$ for some $i$. Since $G\setminus N_G[x]=G_i\setminus N_{G_i}[x]$, we have that $G\setminus N_G[x]$ is Cohen-Macaulay and hence $G$ is Buchsbaum.
   \qed 
   \end{enumerate}
   \end{proof}

  \noindent
  Remark that, for $d=2$, some of the properties in the above proposition are proved in \cite{MM}, Proposition 3.2. Our proof is inspired by the proof of that proposition.

\begin{proposition}\label{joinmultipartite}

Let $t\geq 2$ and $n_1,\dots ,n_t\geq 2$ be integers. If $G$ denotes the complete multipartite graph $K_{n_1,\dots ,n_t}$, then
\begin{enumerate}[(i)]
\item $\mathrm{reg}(R/I(G))=1$, $\mathrm{pd}(R/I(G))=\sum_{i=1}^tn_i-1$ and $\nu(G)=1$.
\item $G$ is well-covered and Buchsbaum. 
\item $G$ does not satisfy any of the following properties: vertex decomposability/shellability/Cohen-Macaulayness/sequentially Cohen-Macaulayness or Serre's condition $S_2$.
\end{enumerate}
\end{proposition}

\begin{proof}
 For (i), note that the statement about regularity follows from Theorem \ref{Jac}. Since $G^c$ is disjoint union of complete graphs $K_{n_i}$, $\mathrm{pd}(R/I(G))=\sum_{i=1}^tn_i-1$ (see Theorem \ref{pd}). Also, $\nu(G)=1$ follows from the definition. As for (ii) and (iii), note that $G=G_1*\dots *G_t$, where $G_i$'s are graphs consisting of $n_i$ number of isolated vertices. Now the statements quickly follows from Proposition \ref{joinproperties}.
 \qed
\end{proof}


\section{{\bf Circulant graphs $C_n(1,\dots ,\widehat j,\dots ,\lfloor\frac{n}{2}\rfloor)$}}\label{jhat}
In this section we study various algebraic and combinatorial properties of the edge ideal associated to the circulant graph $G=C_n(1,\dots ,\widehat j,\dots ,\lfloor\frac{n}{2}\rfloor)$, where $1\leq j\leq \lfloor\frac{n}{2}\rfloor$. We begin with the observation that $G$ can be written as joins of complement of cycles. Then we compute the regularity and Betti numbers using Proposition \ref{joinregularity} and Proposition \ref{joinoftwographs}, respectively. We also determine when these graphs are well-covered, Cohen-Macaulay etc.  

\begin{lemma}\label{Ci}
Let $G=C_n(1,\dots ,\widehat j,\dots ,\lfloor\frac{n}{2}\rfloor)$ with $d=\gcd(j,n)$. Then $G=\underbrace{G_1*\dots * G_1}_{d\text{-times}}$, where $G_1$ is a graph on $\frac{n}{d}$ number of vertices with $G_1^c=C_{\frac{n}{d}}$, the cycle of length $\frac{n}{d}$ (if $\frac{n}{d}=2$, then $C_2$ denotes the complete graph on $2$ vertices).
\end{lemma}

\begin{proof}
Suppose the vertices of $G$ are labelled as $\{0,1,\dots ,n-1\}$. We partition the set $V(G)$ into $d$-components $V_i=\{i, j+i, 2j+i, \dots, (\frac{n}{d}-1)j+i\}$ for $0\leq i<d$, where the indices of the vertices computed modulo $n$. Note that for each $r \neq s$, there is an edge from every vertex of $G_{V_r}$ to the vertices of $G_{V_s}$. In other words $G=G_{V_0}* \dots * G_{V_{d-1}}$. Also if we consider the complement graph $G_{V_i}^c$ on the vertex set $V_i$, then $G_{V_i}^c$ is a cycle of length $\frac{n}{d}$ for each $i$, thus proving the statement. 
\qed \\

\end{proof}

See, for example, the graph $C_{12}(1,\widehat{2},3,4,5,6) \cong C_6^c * C_6^c$ in Figure \ref{figure 1}. 

\begin{figure}[h!]
\centering
\begin{tikzpicture}
[scale=.55]
 \draw [fill] (0,0) circle [radius=0.1];
 \draw [fill] (1.5,0) circle [radius=0.1];
\draw [fill] (2.4,0.8) circle [radius=0.1];
\draw [fill] (2.75,1.8) circle [radius=0.1];
\draw [fill] (2.7,2.9) circle [radius=0.1];
\draw [fill] (2.35,3.8) circle [radius=0.1];
\draw [fill] (1.5,4.6) circle [radius=0.1];
\draw [fill] (0,4.6) circle [radius=0.1];
\draw [fill] (-0.85,3.8) circle [radius=0.1];
\draw [fill] (-1.25,2.9) circle [radius=0.1];
\draw [fill] (-1.25,1.8) circle [radius=0.1];
\draw [fill] (-0.85,0.8) circle [radius=0.1];
\node at (0,-0.6) {$0$};
\node at (1.5,-0.6) {$1$};
\node at (2.9,0.5) {$2$};
\node at (3.25,1.7) {$3$};
\node at (3.2,2.9) {$4$};
\node at (2.87,3.89) {$5$};
\node at (1.9,4.9) {$6$};
\node at (-0.3,4.90) {$7$};
\node at (-1.37,3.9) {$8$};
\node at (-1.75,2.9) {$9$};
\node at (-1.95,1.7) {$10$};
\node at (-1.45,0.5) {$11$};
\draw (0,0)--(1.5,0)--(2.4,0.8)--(2.75,1.8)--(2.7,2.9)--(2.35,3.8)--(1.5,4.6)--(0,4.6)--(-0.85,3.8)--(-1.25,2.9)--(-1.25,1.8)--(-0.85,0.8)--(0,0)--(2.75,1.8)--(1.5,4.6)--(-1.25,2.9)--(0,0)--(2.7,2.9)--(-0.85,3.8)--(0,0)--(2.35,3.8)--(-0.85,3.8)--(-0.85,0.8)--(2.4,0.8)--(2.35,3.8)--(-1.25,1.8);
\draw (2.7,2.9)--(-1.25,1.8)--(2.75,1.8)--(-0.85,3.8)--(1.5,0)--(2.7,2.9)--(0,4.6)--(-1.25,1.8)--(1.5,0)--(1.5,4.6)--(-1.25,1.8)--(2.4,0.8)--(1.5,4.6)--(-0.85,0.8)--(2.75,1.8)--(0,4.6)--(-0.85,0.8)--(2.7,2.9)--(-1.25,2.9)--(1.5,0)--(2.35,3.8)--(-1.25,2.9)--(2.4,0.8)--(0,4.6)--(0,0)--(1.5,4.6);
\draw (2.75,1.8)--(-1.25,2.9);
\draw (1.5,0)--(0,4.6);
\draw (9,1.5) --(11,3.5)--(8,2.5)--(11,1.5)--(9,3.5)--(9,1.5)--(12,2.5)--(8,2.5);
\draw (2.4,0.8)--(-0.85,3.8);
\draw (11,3.5)--(11,1.5);
\draw (9,3.5)--(12,2.5);
\draw (2.35,3.8)--(-0.85,0.8);
\node at (5,2.5) {$\cong$};
\draw [fill] (9,1.5) circle [radius=0.1];
\draw [fill] (11,1.5) circle [radius=0.1];
\draw [fill] (11,3.5) circle [radius=0.1];
\draw [fill] (12,2.5) circle [radius=0.1];
\draw [fill] (9,3.5) circle [radius=0.1];
\draw [fill] (8,2.5) circle [radius=0.1];
\draw (18,1.5) --(20,3.5)--(17,2.5)--(20,1.5)--(18,3.5)--(18,1.5)--(21,2.5)--(17,2.5);
\draw (20,3.5)--(20,1.5);
\draw (18,3.5)--(21,2.5);
\draw [fill] (18,1.5) circle [radius=0.1];
\draw [fill] (20,1.5) circle [radius=0.1];
\draw [fill] (20,3.5) circle [radius=0.1];
\draw [fill] (21,2.5) circle [radius=0.1];
\draw [fill] (18,3.5) circle [radius=0.1];
\draw [fill] (17,2.5) circle [radius=0.1];
\node at (9,0.9) {$0$};
\node at (11,0.9) {$2$};
\node at (12.5,2.5) {$4$};
\node at (11.5,3.8) {$6$};
\node at (8.5,3.8) {$8$};
\node at (7.3,2.5) {$10$};
\node at (18,0.9) {$1$};
\node at (20,0.9) {$3$};
\node at (21.5,2.5) {$5$};
\node at (20.5,3.8) {$7$};
\node at (17.5,3.8) {$9$};
\node at (16.3,2.5) {$11$};
\node at (14.5,2.5) {\LARGE $*$};
\end{tikzpicture}
\caption{$C_{12}(1,\widehat 2,3,4,5,6) \cong C_6^c * C_6^c$}\label{figure 1}
\end{figure}
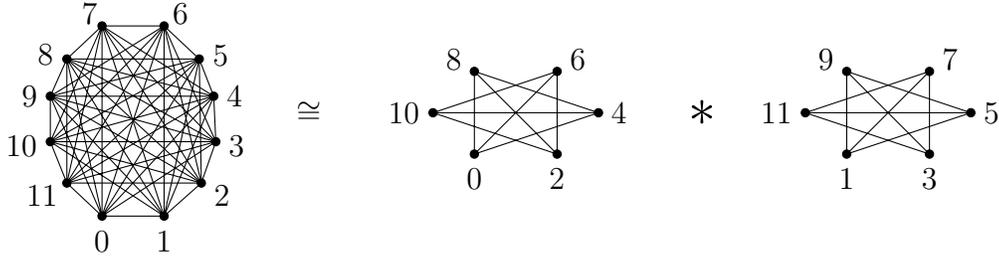

Let $\frac{n}{d}=k$. As $j\leq\frac{n}{2}$, clearly $d<n$ and hence $k\geq 2$. By Proposition \ref{joinregularity}, in order to determine $\mathrm{reg}(R/I(G))$, it is enough to find $\mathrm{reg}(R/I(G_1))$, where $G_1$ is a graph on $k$ vertices with $G_1^c=C_k$, the cycle of length $k$. We compute $\mathrm{reg}(R/I(G_1))$ by using the well-known Hochster's formula.

\begin{theorem}\label{regCi}
Let $G_1$ be a graph on $k$ number of vertices such that $G_1^c$ is a cycle of length $k\geq 4$. Then 
  $\mathrm{reg}\left(R/I(G_1)\right)=
  2$
  
\end{theorem}

\begin{proof}
Let $\Delta_{G_1}$ be the independence complex of $G_1$ on the vertex set $\{0,1,\ldots ,k-1\}$. For each $k\geq 4$, the facets of $\Delta_{G_1}$ are \[\{0,1\},\{1,2\},\ldots, \{k-2,k-1\},\{0,k-1\}.\]
Recall that the Betti numbers $\beta_{i,j}(R/I(G_1))$ are also denoted by $\beta_{i,j}(G_1)$. We have by Hochster's formula, 
  \begin{align}\label{Hochster formula}
  \beta_{i,r}\left(G_1\right)=\sum_{\underset{|V|=r}{V\subseteq\{0,1,\ldots ,k-1\}}}\dim_{\mathbb K}\widetilde H_{|V|-i-1}(\Delta[V];\mathbb K),
  \end{align}
  where $\Delta[V]=\{\tau\in\Delta_{G_1}|\tau\subseteq V\}$ is a subcomplex of $\Delta_{G_1}$. Since $\Delta_{G_1}$ is a one dimensional simplicial complex, $\widetilde H_{|V|-i-1}$ is possibly non-zero for the cases $|V|=i+1$ and $|V|=i+2$, i.e., $r=i+1$ and $r=i+2$, respectively. Therefore, we have possibly non-zero Betti numbers $\beta_{i,i+1}$ for $i=0,1,\ldots ,k-1$ and $\beta_{i,i+2}$ for $i=2,\ldots ,k-2$.\\
\noindent
  {\bf Claim}: $\beta_{i,i+2}(G_1)=\begin{cases}
1\quad\text{if}\quad i=k-2, \\
0 \quad\text{otherwise}.
  \end{cases}
  $\\
  \noindent
  {\bf Proof of the claim}: Note that by Hochster's formula
  \[\beta_{i,i+2}\left(G_1\right)=\sum_{\underset{|V|=i+2}{V\subseteq\{0,1,\ldots ,k-1\}}}\dim_{\mathbb K}\widetilde H_{1}(\Delta[V];\mathbb K).\]
  Clearly, $\widetilde{H}_{1}(\Delta[V];\mathbb K)=0$ if $|V|\neq k$ as in this case the connected components of $\Delta[V]$ are contractible.  Now, if $|V|=k$, i.e., $i+2=k$, then $\Delta[V]=\Delta_{G_1}$.\\ 
  The claim follows by noting that $\Delta_{G_1}$ is the triangulation of the $1$-dimensional sphere $\mathbb{S}^1$. Consequently, $\mathrm{reg}(R/I(G_1))=2$ for $k\geq 4$. \qed
  \end{proof}
  
\begin{corollary}\label{regcnhat}
Let $G=C_n(1,\ldots ,\widehat j,\ldots ,\lfloor\frac{n}{2}\rfloor)$ and $d=\gcd(j,n)$. Then 
    \[\mathrm{reg}\left(\frac{R}{I(G)}\right)= \begin{cases}
  1\quad  \mathrm{if}\quad n=2j \text{ or } n=3d \\
  2\quad \hspace{0.3 cm} \quad \text{otherwise.}
  \end{cases}
  \]
   \end{corollary}

  \begin{proof}
    If $\frac{n}{d}=2$ or $3$ then $n=2j$ or $n=3d$ respectively and in these cases, by Lemma \ref{Ci}, the graph $G$ is a multipartite ($d$-partite) graph with each partition set having $\frac{n}{d}$ number of vertices. Therefore by Theorem\ref{Jac}, $\operatorname{reg}(R/I(G))=1$ if $n=2j$ or $n=3d$. If $k=\frac{n}{d}\geq 4$ then $\operatorname{reg}(R/I(G))=2$ by Lemma \ref{Ci}, Proposition \ref{joinregularity} and Theorem \ref{regCi}.
   \qed
   \end{proof}

 \begin{remark} 
  Since $\mathrm{reg}(I(G))=\mathrm{reg}(R/I(G))+1$, Corollary \ref{regcnhat} and \cite{VTUP},Theorem 3.3 are essentially equivalent. Here we give a slightly different proof. 
 \end{remark}
  
  In the next proposition, we show that the Betti numbers $\beta_{i,i+1}(G_1)$ are palindromic in the following sense:
\begin{proposition}
Let $k\geq 4$. We have $\beta_{i,i+1}(G_1)=0$ if $i\notin\{1,2,\dots ,k-3\}$ and $\beta_{i,i+1}(G_1)=\beta_{k-i-2,k-i-1}(G_1)$ for $1\leq i\leq k-3$.
\end{proposition}

\begin{proof}
We make use of Hochster's formula (\ref{Hochster formula}) once again. Clearly, $\beta_{0,1}(G_1)=0$ as $\dim_{\mathbb K}\widetilde H_0(\Delta[V];\mathbb K)=0$ in this case. 
If $i\geq k-2$, then $|V|\geq k-1$ and hence $\Delta[V]$ has only one connected component. Therefore, 
  \[\beta_{i,i+1}\left(G_1\right)=0\quad\text{for}\quad i\geq k-2,\]
  and this proves the first part of the proposition. For the second part, by Hochster's formula, we have
  \[\beta_{i,i+1}\left(G_1\right)=\sum_{\underset{|V|=i+1}{V\subseteq\{0,1,\dots,k-1\}}}\dim_{\mathbb K}\widetilde H_{0}(\Delta[V];\mathbb K)\]
  and,
  \[\beta_{k-i-2,k-i-1}\left(G_1\right)=\sum_{\underset{|V|=k-i-1}{V\subseteq\{0,1,\dots,k-1\}}}\dim_{\mathbb K}\widetilde H_{0}(\Delta[V];\mathbb K),\]
  for $1\leq i\leq k-3$. Note that if $|V|=i+1$, then $|V^c|=k-i-1$ and vice-versa. Since 
  $\dim_{\mathbb K}\widetilde H_{0}(\Delta[V];\mathbb K)=\text{(number of connected components of }\Delta[V])\hspace{0.01cm}-1$,
  we just need to show that number of connected components of $V$ with $|V|=i+1$ is same as number of connected components of $V^c$ and this is clear from the structure of $\Delta_{G_1}$. \qed \\ 
  
\end{proof}
  

  In the last part of this section we would like to calculate the Betti numbers of the circulant graph $G$ in Lemma \ref{Ci}. But first we state below the Betti numbers for the graph $G_1$.
\begin{proposition}\label{BettiCyclecomplement}
  For $k\geq 2$, the Betti numbers of $R/I(G_1)$ are given by
  \begin{align}\label{firstBetti}
   \beta_{i,i+1}(G_1)=
   \begin{cases}
    \binom{k}{i+1}\frac{i(k-i-2)}{k-1}+1\quad\text{if}\hspace{.3cm} i=k-1, \\
    \binom{k}{i+1}\frac{i(k-i-2)}{k-1}\hspace{1.1cm}\text{otherwise}.
   \end{cases}
   \end{align}
  and,
  \begin{align}\label{secondBetti}
    \beta_{i,i+2}(G_1)=
    \begin{cases}
     1\quad\text{if}\hspace{.3cm}k\geq 4\text{ and } i=k-2, \\
     0\quad\text{otherwise}.
    \end{cases}
   \end{align}
\end{proposition}
\begin{proof}
   The Betti numbers $\beta_{i,i+2}$ are computed in the above claim (in proof of Theorem \ref{regCi}). As for the Betti numbers $\beta_{i,i+1}$, the formula is given in \cite{Do}, Remark 7, for $k\geq 5$. For $k=4$, it can be checked using Macaulay2 \cite{MAC2}. Note that, $I(G_1)$ is the zero ideal when $k=2$ or $3$.
   \qed \\ 
   
   \end{proof}
  

  \begin{theorem}\label{Bettijoindtimes}
   Let $G_1$ be the complement of a cycle of length $k\geq 2$. Then for $d\geq 2$,
   \begin{align}\label{firstBettinumbers}
    \beta_{i,i+1}(\underbrace{G_1*G_1*\dots * G_1}_{d\text{-times}})&=\binom{dk}{i+1}\frac{i(dk-i-2)}{dk-1}+d\binom{(d-1)k}{i-k+1}
    \end{align}
   and
   \begin{align}\label{secondBettinumbers}
    \hspace{-2.9cm}\beta_{i,i+2}(\underbrace{G_1*G_1*\dots * G_1}_{d\text{-times}})&=\begin{cases}                                                                                       
                           d\binom{(d-1)k}{i-k+2}\quad\text{if}\quad k\geq 4, \\
     0\hspace{1.7cm}\text{otherwise.}                                                           
    \end{cases}
   \end{align}
\end{theorem}
\begin{proof}
   The following identities can be verified directly:
\begin{align}\label{secondidentity}
  \sum_{j=0}^t\binom{u}{j}\binom{v}{t-j}=\binom{u+v}{t},
 \end{align}
   \begin{align}\label{expand}
    \binom{m}{i+1}\frac{i(m-i-2)}{m-1}=m\binom{m-1}{i}-m\binom{m-2}{i-1}-\binom{m}{i+1}.
   \end{align}
   We prove the formulas of Betti numbers by induction on $d$. We have by Proposition \ref{joinoftwographs},
    \begin{align*}
   \begin{aligned}
    &\beta_{i,i+1}(G_1*G_1)\\
    &=2\sum_{j=0}^{i-1}\binom{k}{j}\beta_{i-j,i-j+1}(G_1)+\sum_{j=1}^i\binom{k}{j}\binom{k}{i-j+1}\\
    &=2\sum_{j=0}^{i-1}\binom{k}{j}\binom{k}{i-j+1}\frac{(i-j)(k-i+j-2)}{k-1}+2\binom{k}{i-k+1}+\sum_{j=1}^i\binom{k}{j}\binom{k}{i-j+1}\hspace{0.2cm}\text{(by (\ref{firstBetti}))}\\
    &=2\sum_{j=0}^{i-1}\binom{k}{j}\left[k\binom{k-1}{i-j}-k\binom{k-2}{i-j-1}-\binom{k}{i-j+1}\right]+2\binom{k}{i-k+1}\\
    &\hspace{11.3cm}+\sum_{j=1}^i\binom{k}{j}\binom{k}{i-j+1}\quad\text{(by (\ref{expand}))}.\\
   \end{aligned}
   \end{align*}
    
    \noindent
    Using Equation \ref{secondidentity}, we get $\sum_{j=0}^{i-1}\binom{k}{j}\binom{k-1}{i-j}=\binom{2k-1}{i}-\binom{k}{i}$ and $\sum_{j=0}^{i-1}\binom{k}{j}\binom{k-2}{i-j-1}=\binom{2k-2}{i-1}$. Also, $\sum_{j=0}^{i-1}\binom{k}{j}\binom{k}{i-j+1}=\binom{2k}{i+1}-k\binom{k}{i}-\binom{k}{i+1}$ and $\sum_{j=1}^{i}\binom{k}{j}\binom{k}{i-j+1}=\binom{2k}{i+1}-2\binom{k}{i+1}$. Hence,
    \begin{align*}
   \beta_{i,i+1}(G_1*G_1)&=2k\binom{2k-1}{i}-2k\binom{2k-2}{i-1}-\binom{2k}{i+1}+2\binom{k}{i-k+1}\\
   &=\binom{2k}{i+1}\frac{i(2k-i-2)}{2k-1}+2\binom{k}{i-k+1}.
   \end{align*}
   Therefore, Equation (\ref{firstBettinumbers}) is valid for $d=2$. Assuming the formula is true for $d\geq 2$, we verify it for $d+1$. Once again, Proposition \ref{joinoftwographs} gives
    \begin{align*}
    &\beta_{i,i+1}(\underbrace{G_1*G_1*\dots *G_1}_{(d+1)\text{-times}})\\
    &=\sum_{j=0}^{i-1}\binom{k}{j}\beta_{i-j,i-j+1}(\underbrace{G_1*G_1*\dots *G_1}_{d\text{-times}})+\sum_{j=0}^{i-1}\binom{dk}{j}\beta_{i-j,i-j+1}(G_1)\\
    &\hspace{11.3cm}+\sum_{j=1}^i\binom{dk}{j}\binom{k}{i-j+1}\\
    &=\sum_{j=0}^{i-1}\binom{k}{j}\left[\binom{dk}{i-j+1}\frac{(i-j)(dk-i+j-2)}{dk-1}+d\binom{(d-1)k}{i-j-k+1}\right]\\
    &\hspace{.3cm}+\sum_{j=0}^{i-1}\binom{dk}{j}\binom{k}{i-j+1}\frac{(i-j)(k-i+j-2)}{k-1}+\binom{dk}{i-k+1}+\sum_{j=1}^i\binom{dk}{j}\binom{k}{i-j+1}.
   \end{align*}
   By using Equation \ref{expand} and \ref{secondidentity} in a similar way as for the $d=2$ case we can show that
   \begin{align*}
    \beta_{i,i+1}(\underbrace{G_1*G_1*\dots *G_1}_{(d+1)\text{-times}})=\binom{(d+1)k}{i+1}\frac{i((d+1)k-i-2)}{(d+1)k-1}+(d+1)\binom{dk}{i-k+1}.
   \end{align*}
   \noindent
   This completes the induction.
   \noindent
   Equation (\ref{secondBettinumbers}) can also be verified by induction in a similar way. \qed \\ 
   
 \end{proof}
   


  
  
Lemma \ref{Ci},  Corollary \ref{regcnhat}, Proposition \ref{BettiCyclecomplement}, and Theorem \ref{Bettijoindtimes} all together yields the main result of this section.
  \begin{theorem}
 Suppose $n\geq 5$ is an integer. Let $G=C_n(1,\dots,\widehat j,\dots ,\lfloor\frac{n}{2}\rfloor)$ with $d=\gcd(j,n)$ and $k=\frac{n}{d}$.
  {\bf Case I}: For $d=1$,
  \[
  \beta_{i,i+1}(G)=
  \begin{cases}
   \binom{n}{i+1}\frac{i(n-i-2)}{n-1}\quad\text{for $1\leq i\leq n-2$,}\\
   0\hspace{2.3cm} \text{otherwise,}
  \end{cases}
  \]
   \begin{align*}
    \hspace{-2cm}\beta_{i,i+2}(G)=
    \begin{cases}
     1\quad\text{if $i=n-2$,} \\
     0\quad\text{otherwise.}
    \end{cases}
   \end{align*}
   {\bf Case II}:  For $d\geq 2$,
   \[
    \beta_{i,i+1}(G)=\binom{n}{i+1}\frac{i(n-i-2)}{n-1}+d\binom{n-k}{i-k+1},
   \]
  \[
   \hspace{-2cm}\beta_{i,i+2}(G)=
    \begin{cases}
     d\binom{n-k}{i-k+2} \quad \text{if  $n\geq 4d$,} \\
     0\hspace{1.8cm} \text{otherwise.}
    \end{cases}
  \]
   \end{theorem}
  
  We now describe various combinatorial properties of the circulant graph $G$.
  
  \begin{theorem}\label{hatproperites}
    Let $n\geq 4$ be an integer and $G=C_n(1,\dots,\widehat j,\dots ,\lfloor\frac{n}{2}\rfloor)$. Then 
   \begin{enumerate}[(i)]
   \item The induced matching number $\nu(G)=\begin{cases}
                                              2\quad\text{if }k=4,\\
                                              1\quad\text{otherwise,}
                                             \end{cases}
$ 
\newline
where $k=\frac{n}{\gcd(n,j)}$.
   \item $G$ is well-covered as well as a Buchsbaum graph.
   \item $G$ is vertex decomposable/shellable/Cohen-Macaulay/sequentially Cohen-Macaulay (or $S_2$) if and only if $\gcd(n,j)=1$.
   \end{enumerate}
  \end{theorem}
  \begin{proof}
   Let $\gcd(n,j)=d$ and $k=\frac{n}{d}\geq 2$. Recall that by Lemma \ref{Ci}, $G=\underbrace{G_1*\dots * G_1}_{d\text{-times}}$, where $G_1$ is a graph on $k$ number of vertices with $G_1^c=C_{k}$, the cycle of length $k$.
   First consider the case $d=1$. 
  In that case $n=k$ and $G=G_1$. Let $V(C_k)=\{0,1,\dots ,k-1\}$. Then the facets of the simplicial complex $\Delta_{G_1}$ are $\{0,1\},\{1,2\},\dots ,\{k-1,0\}$.
   \begin{enumerate}[(i)]
   \item When $n=4$, we have $\nu(G)=\nu(C_4^c)=2$. For $n\geq 5$, the fact that $\nu(G)=\nu(C_n^c)=1$ follows from a direct inspection of the structure of $C_n^c$.
   \item $G_1$ is well-covered as all maximal independent sets have cardinality $2$. Also for $x\in V(G_1)$, $G_1\setminus N_{G_1}[x]$ is the complete graph on $2$ vertices and hence Cohen-Macaulay. Therefore $G_1$ is Buchsbaum.
   \item $G_1$ is vertex decomposable, shellable and Cohen-Macaulay as $\Delta_{G_1}$ is a pure $1$-dimensional connected simplicial complex (see Theorem \ref{propeties results}). Since Cohen-Macaulay simplicial complexes of dimension $1$ are also sequentially Cohen-Macaulay, $G_1$ is sequentially Cohen-Macaulay. The $S_2$ property follows from the definition.
   \end{enumerate}
   We now consider the case $d\geq 2$. If $k=2$ or $3$, $G$ is a multipartite graph and hence $\nu(G)=1$. For $k\geq 4$, the result is deduced by applying Proposition \ref{joinproperties}. Also by Proposition \ref{joinproperties}, $G$ is well-covered/Buchsbaum if and only if $G_1$ is well-covered/Buchsbaum. Recall that $G_1^c$ is a cycle of length $k$. If $k\geq 4$ then the statement in (ii) follows from the $d=1$ case. For $k=2$ and $3$, $G_1$ consists of isolated vertices and hence $G_1$ is well-covered as well as Buchsbaum.
   
   For $G$ to be vertex decomposable/shellable/Cohen-Macaulay/sequentially Cohen-Macaulay (or $S_2$), $G_1$ needs to be a complete graph (by Proposition \ref{joinproperties}). But $G_1$ can never be a complete graph and this completes the proof of the proposition. \qed \\ 
   
\end{proof}
  Remark that for $S_2$ property the statement in Proposition \ref{hatproperites} is proved in \cite{AM1}, Theorem 4.1. Also except the $S_2$ and sequentially Cohen-Macaulay properties and the induced matching number the statements for all other properties have been proved in \cite{EMT}, Theorem 4.2. Here we give an alternative proof using Proposition\ref{joinproperties} and Lemma \ref{Ci}.


\section{{\bf Circulant graphs $C_{lm}(1,\dots ,\widehat{2l},\dots ,\widehat{3l},\dots ,\lfloor\frac{lm}{2}\rfloor)$}}\label{power}
In this section we study the circulant graph $C_{n}(1,\dots ,\widehat{2l},\dots ,\widehat{3l},\dots ,\lfloor\frac{n}{2}\rfloor)$,  where $n=lm$ is a composite number. We first show that this circulant graph can be written as join of cycles. More generally, we determine when a product of cycles is a circulant graph. Using this structure result we compute the Betti numbers. Various combinatorial properties associated to the ideal are also determined.

\begin{lemma}\label{joingeneral}
  Let $m_1,\dots ,m_l\geq 3$ be integers. Then $C_{m_1}*\dots *C_{m_l}$ is a circulant graph if and only if $m_1=\dots =m_l=m$ for some integer $m\geq 3$. In addition, if this is the case, then $\underbrace{C_m*\dots *C_m}_{l\text{-times}}=C_{lm}(1,\dots ,\widehat{2l},\dots ,\widehat{3l},\dots ,\lfloor{\frac{lm}{2}\rfloor})$.  
\end{lemma}

\begin{proof}
First notice that if $m_i\neq m_j$ for some $i\neq j$, then $C_{m_1}*\dots *C_{m_l}$ is not a regular graph (in regular graph all vertices have same degree) and hence cannot be a circulant graph. Now we show that $\underbrace{C_m*\dots *C_m}_{l\text{-times}}$ is the circulant graph $G=C_{lm}(1,\dots ,\widehat{2l},\dots ,\widehat{3l},\dots ,\lfloor{\frac{lm}{2}\rfloor})$. We partition the set $V(G)$ into $l$-components:
 $V_i=\{i, l+i, \dots, (m-1)l+i\}$, for $0\leq i < l$. Note that, the induced subgraphs $G_{V_i}$ are the cycles $C_m$ for all $i$. Also, for all $i,j$ with $i\neq j$ there is an edge between each vertex of $V_i$ and $V_j$, thus proving the statement.
\qed \\ 

\end{proof}

See, for example, the graph $C_{18}(1,2,3,4,5,\widehat{6},7,8,\widehat{9}) \cong C_6 * C_6 * C_6$ in Figure \ref{figure 6}.

\begin{figure}[ht!]
\centering
\begin{tikzpicture}
[scale=0.75]
\draw [fill] (0,0) circle [radius=0.1];
\draw [fill] (1,0.2) circle [radius=0.1];
\draw [fill] (-1,0.2) circle [radius=0.1];
\draw [fill] (1.8,0.7) circle [radius=0.1];
\draw [fill] (-1.8,0.7) circle [radius=0.1];
\draw [fill] (2.6,1.7) circle [radius=0.1];
\draw [fill] (-2.6,1.7) circle [radius=0.1];
\draw [fill] (3,2.7) circle [radius=0.1];
\draw [fill] (-3,2.7) circle [radius=0.1];
\draw [fill] (3,3.8) circle [radius=0.1];
\draw [fill] (-3,3.8) circle [radius=0.1];
\draw [fill] (2.6,4.9) circle [radius=0.1];
\draw [fill] (-2.6,4.9) circle [radius=0.1];
\draw [fill] (1.8,5.9) circle [radius=0.1];
\draw [fill] (-1.8,5.9) circle [radius=0.1];
\draw [fill] (1,6.4) circle [radius=0.1];
\draw [fill] (-1,6.4) circle [radius=0.1];
\draw [fill] (0,6.6) circle [radius=0.1];
\draw (-2.6,4.9)--(3,2.7);
\draw (-1.8,0.7)--(0,6.6);
\draw (0,6.6)--(3,2.7);
\draw (-1.8,0.7)--(3,3.8);
\draw (3,3.8)--(-2.6,4.9);
\draw (-2.6,4.9)--(1,0.2)--(-1.8,5.9)--(2.6,1.7)--(-3,3.8)--(3,3.8)--(-2.6,1.7)--(1.8,5.9)--(-1,0.2)--(0,6.6)--(1,0.2)--(-1,0.2)--(-2.6,1.7)--(-3,3.8)--(-1.8,5.9)--(0,6.6)--(1.8,5.9)--(3,3.8)--(2.6,1.7)--(1,0.2)--(1.8,0.7)--(2.6,1.7)--(3,2.7)--(3,3.8)--(2.6,4.9)--(1.8,5.9)--(1,6.4)--(0,6.6)--(-1,6.4)--(-1.8,5.9)--(-2.6,4.9)--(-3,3.8)--(-3,2.7)--(-2.6,1.7)--(-1.8,0.7)--(-1,0.2)--(0,0)--(1,0.2)--(3,3.8)--(0,6.6)--(-3,3.8)--(-1,0.2)--(2.6,1.7)--(1.8,5.9)--(-1.8,5.9)--(-2.6,1.7)--(1,0.2)--(2.6,4.9)--(-1.8,5.9)--(-1.8,0.7)--(2.6,1.7)--(1,6.4)--(-3,3.8)--(0,0)--(3,3.8)--(-1,6.4)--(-2.6,1.7)--(1.8,0.7)--(1.8,5.9)--(-2.6,4.9)--(-1,0.2)--(3,2.7)--(1.8,0.7)--(0,6.6)--(-3,2.7)--(1,0.2);
\draw (1,0.2)--(1,6.4)--(-2.6,1.7)--(3,2.7)--(-1.8,5.9)--(0,0)--(1.8,5.9)--(-3,2.7)--(2.6,1.7)--(-1,6.4)--(-1,0.2)--(2.6,4.9)--(-3,3.8)--(1.8,0.7)--(2.6,4.9)--(-1,6.4)--(-3,2.7)--(0,0)--(3,2.7)--(1,6.4)--(-2.6,4.9)--(-1.8,0.7)--(1.8,0.7);
\draw (1.8,0.7)--(-1,6.4)--(0,0)--(1,6.4)--(-1.8,0.7)--(2.6,4.9)--(-3,2.7);
\draw (1.8,0.7)--(-2.6,4.9);
\draw (1,6.4)--(2.6,4.9)--(3,2.7)--(1.8,0.7)--(0,0)--(-1.8,0.7)--(-3,2.7)--(-2.6,4.9)--(-1,6.4)--(1,6.4);
\draw (3,2.7)--(-3,2.7);
\draw (0,0)--(2.6,1.7)--(2.6,4.9)--(0,6.6)--(-2.6,4.9)--(-2.6,1.7)--(0,0);
\draw (1,0.2)--(3,2.7)--(1.8,5.9)--(-1,6.4)--(-3,3.8)--(-1.8,0.7)--(1,0.2);
\draw (1.8,0.7)--(3,3.8)--(1,6.4)--(-1.8,5.9)--(-3,2.7)--(-1,0.2)--(1.8,0.7);
\node at (0,-0.5) {$0$};
\node at (1.3,-0.3) {$1$};
\node at (2.3,0.3) {$2$};
\node at (3.0,1.3) {$3$};
\node at (3.5,2.6) {$4$};
\node at (3.5,3.8) {$5$};
\node at (3.1,5.0) {$6$};
\node at (2.3,6.1) {$7$};
\node at (1.3,6.9) {$8$};
\node at (0,7.1) {$9$};
\node at (-1.3,6.9) {$10$};
\node at (-2.4,6.1) {$11$};
\node at (-3.2,5.0) {$12$};
\node at (-3.6,3.8) {$13$};
\node at (-3.6,2.6) {$14$};
\node at (-3.1,1.3) {$15$};
\node at (-2.4,0.3) {$16$};
\node at (-1.4,-0.3) {$17$};
\node at (4.6,3.3) {{\Large $\cong$}};
\draw [fill] (6.5,2.3) circle [radius=0.1];
\draw [fill] (7.5,2.3) circle [radius=0.1];
\draw [fill] (8,3.3) circle [radius=0.1];
\draw [fill] (7.5,4.3) circle [radius=0.1];
\draw [fill] (6.5,4.3) circle [radius=0.1];
\draw [fill] (6,3.3) circle [radius=0.1];
\draw (6.5,2.3)--(7.5,2.3)--(8,3.3)--(7.5,4.3)--(6.5,4.3)--(6,3.3)--(6.5,2.3);
\node at (6.5,1.8) {$0$};
\node at (7.5,1.8) {$3$};
\node at (8.2,2.8) {$6$};
\node at (7.9,4.5) {$9$};
\node at (6.0,4.5) {$12$};
\node at (5.7,2.9) {$15$};
\node at (9,3.3) {{\LARGE $*$}};
\draw [fill] (10.5,2.3) circle [radius=0.1];
\draw [fill] (11.5,2.3) circle [radius=0.1];
\draw [fill] (12,3.3) circle [radius=0.1];
\draw [fill] (11.5,4.3) circle [radius=0.1];
\draw [fill] (10.5,4.3) circle [radius=0.1];
\draw [fill] (10,3.3) circle [radius=0.1];
\draw (10.5,2.3)--(11.5,2.3)--(12,3.3)--(11.5,4.3)--(10.5,4.3)--(10,3.3)--(10.5,2.3);
\node at (10.5,1.8) {$1$};
\node at (11.5,1.8) {$4$};
\node at (12.2,2.8) {$7$};
\node at (11.9,4.5) {$10$};
\node at (10.0,4.5) {$13$};
\node at (9.7,2.9) {$16$};
\node at (13,3.3) {{\LARGE $*$}};
\draw [fill] (14.5,2.3) circle [radius=0.1];
\draw [fill] (15.5,2.3) circle [radius=0.1];
\draw [fill] (16,3.3) circle [radius=0.1];
\draw [fill] (15.5,4.3) circle [radius=0.1];
\draw [fill] (14.5,4.3) circle [radius=0.1];
\draw [fill] (14,3.3) circle [radius=0.1];
\draw (14.5,2.3)--(15.5,2.3)--(16,3.3)--(15.5,4.3)--(14.5,4.3)--(14,3.3)--(14.5,2.3);
\node at (14.5,1.8) {$2$};
\node at (15.5,1.8) {$5$};
\node at (16.2,2.8) {$8$};
\node at (15.9,4.5) {$11$};
\node at (14.0,4.5) {$14$};
\node at (13.7,2.9) {$17$};
\end{tikzpicture}
\caption{$C_{18}(1,2,3,4,5,\widehat 6,7,8,\widehat 9) \cong C_6 * C_6 *C_6$}\label{figure 6}
\end{figure}
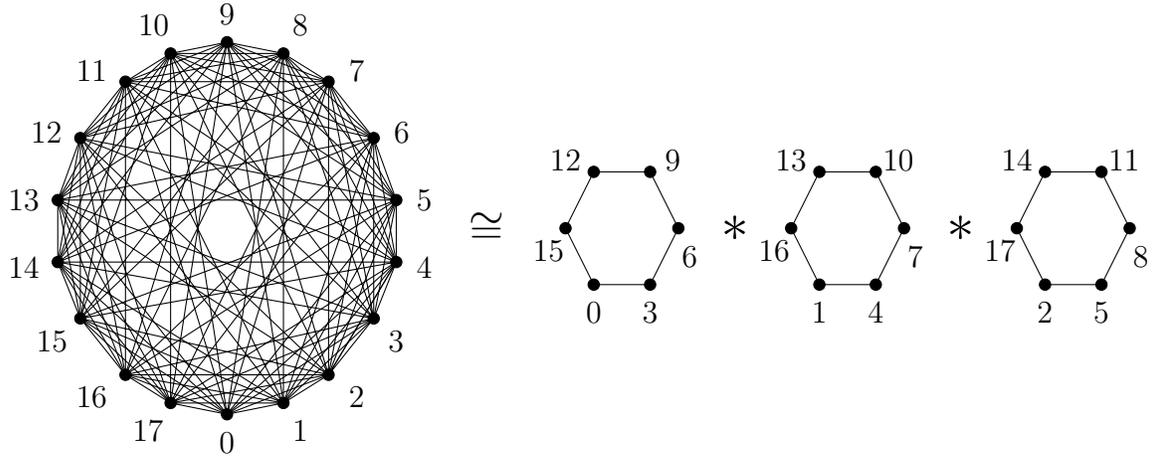

In this section $G$ denotes the circulant graph in Lemma \ref{joingeneral}. Next we compute the Betti numbers $\beta_{i,j}(G)$. But we first need the following lemma which describe algebraic properties of the edge ideal of a cycle.

\begin{lemma}\label{cyclelemma}
Let $m\geq 5$ be an integer. Then 
\begin{enumerate}[(i)]
 \item  $\mathrm{pd}(C_m)=\lfloor\frac{2m+1}{3}\rfloor$ and $\mathrm{reg}(C_m)=\lfloor\frac{m+1}{3}\rfloor$ so that $\mathrm{pd}(C_m)+\mathrm{reg}(C_m)=m$. \\
 \item The initial Betti numbers 
\begin{align*}
         \beta_{i,i+1}(C_m)=\begin{cases}
                             m\quad\text{if $i=1,2$,} \\
                             0\quad \text{if $i>2$.}
                            \end{cases}
\end{align*}
\item  For $2\leq r<\mathrm{reg}(C_m)$ the nonzero Betti numbers
\begin{align*}
 \beta_{i,i+r}(C_m)=\frac{m}{m-2r}\binom{r}{i-r}\binom{m-2r}{r}.
\end{align*}
\item Let $r=\mathrm{reg}(C_m)$ and $p=\mathrm{pd}(C_m)$.

\begin{enumerate}[(a)]
\item For $m\equiv 0 \pmod 3$ the nonzero Betti numbers
\begin{align*} 
\beta_{i,i+r}(C_m)=\begin{cases}
                     3\binom{r}{i-r}\hspace{1.15cm}\text{if } i\neq p,\\
                     3\binom{r}{i-r}-1\quad\text{otherwise}.
                    \end{cases}
\end{align*}

\item For $m\equiv 1 \pmod 3$ the nonzero Betti numbers
\begin{align*}
 \beta_{i,i+r}(C_m)=\begin{cases}
                     m\binom{r}{i-r}\hspace{1.15cm}\text{if } i\neq p,\\
                     m\binom{r}{i-r}+1\quad\text{otherwise}.
                    \end{cases}
\end{align*}

\item For $m\equiv 2 \pmod 3$ the nonzero Betti numbers
\begin{align*}
 \beta_{i,i+r}(C_m)=\begin{cases}
                     1\hspace{.5cm}\text{if } i= p,\\
                     0\quad\text{otherwise}.
                    \end{cases}
\end{align*}
\end{enumerate}
\end{enumerate}
\end{lemma}

\begin{proof}
Follows from \cite{SJ}, Theorem 7.6.28. See also \cite{MM}, Corollary 4.4, Remark 4.5 and Lemma 4.7.
\qed \\ 

\end{proof}

We proceed to compute the initial Betti numbers $\beta_{i,i+1}(G)$.

\begin{theorem}\label{joinBetti1}
Let $m\geq 5$ and $l\geq 2$ be integers. Then  
 \[\beta_{i,i+1}(G)=lm\binom{(l-1)m+1}{i-1}+(l-1)\binom{lm}{i+1}-l\binom{(l-1)m}{i+1}.\].
\end{theorem}
\begin{proof}
Note that $C_{lm}(1,\dots ,\widehat{2l},\dots ,\widehat{3l},\dots ,\lfloor{\frac{lm}{2}\rfloor})=\underbrace{C_m*\dots *C_m}_{l\text{-times}}$. The following identity can be verified directly:
\begin{align}\label{firstidentity}
  \binom{u}{t}+\binom{u}{t-1}=\binom{u+1}{t}.
 \end{align}
We prove the above expression by induction on $l$. For $l=2$, 
\begin{align*}
&\beta_{i,i+1}(C_m*C_m)\\
&=2\sum_{j=0}^{i-1}\binom{m}{j}\beta_{i-j,i-j+1}(C_m)+\sum_{j=1}^i\binom{m}{j}\binom{m}{i-j+1}\hspace{4.2cm}\text{(by Proposition \ref{joinoftwographs})}\\
&=2m\left[\binom{m}{i-1}+\binom{m}{i-2}\right]+\sum_{j=0}^{i+1}\binom{m}{j}\binom{m}{i+1-j}-2\binom{m}{i+1}\hspace{2cm}\text{(by Lemma \ref{cyclelemma} (ii))}\\
&=2m\binom{m+1}{i-1}+\binom{2m}{i+1}-2\binom{m}{i+1}\hspace{7cm}(\text{by }(\ref{firstidentity}) \text{and} (\ref{secondidentity})).
\end{align*}
Therefore, the statement is true for $l=2$. Assuming it is true for $l\geq 2$, we prove it for $l+1$.
Now Proposition \ref{joinoftwographs} yields
\begin{align*}
\begin{aligned}
&\beta_{i,i+1}(\underbrace{C_m*\dots *C_m}_{l+1\text{-times}}) \\
&=\sum_{j=0}^{i-1}\left[\binom{m}{j}\beta_{i-j,i-j+1}(\underbrace{C_m*\dots *C_m}_{l\text{-times}})+\binom{lm}{j}\beta_{i-j,i-j+1}(C_m)\right]+\sum_{j=1}^i\binom{lm}{j}\binom{m}{i-j+1}\\
&=\sum_{j=0}^{i-1}\binom{m}{j}\left[lm\binom{(l-1)m+1}{i-j-1}+(l-1)\binom{lm}{i-j+1}-l\binom{(l-1)m}{i-j+1}\right]\\
&\hspace{.5cm}+m\left[\binom{lm}{i-1}+\binom{lm}{i-2}\right]+\sum_{j=0}^{i+1}\binom{lm}{j}\binom{m}{i-j+1}-\binom{m}{i+1}-\binom{lm}{i+1}\quad\text{(by Lemma \ref{cyclelemma}(ii))}.\\
\end{aligned}
\end{align*}
Again applying the identities (\ref{firstidentity}) and (\ref{secondidentity}) the induction is complete.
\qed \\ 

\end{proof}

We next determine the nonlinear Betti numbers $\beta_{i,i+r}(G)$, where $2\leq r < \mathrm{reg}(G)$. Note that, by Lemma \ref{joingeneral} and  Proposition \ref{joinregularity},  $\mathrm{reg}(G)=\mathrm{reg}(C_m)$.

\begin{theorem}\label{joinBetti2}
Let $m\geq 5$ and $l\geq 2$ be integers. Then for $2\leq r<\mathrm{reg}(G)$,
\[\beta_{i,i+r}(G)=\frac{lm}{m-2r}\binom{m-2r}{r}\binom{(l-1)m+r}{i-r}.\]
\end{theorem}
\begin{proof}
Proof is by induction on $l$. For $l=2$, we use Proposition \ref{joinoftwographs} to get
\begin{align*}
\beta_{i,i+r}(C_m*C_m)&=2\sum_{j=0}^{i+r-2}\binom{m}{j}\beta_{i-j,i-j+r}(C_m)\\
&=2\sum_{j=0}^{i-r}\binom{m}{j}\frac{m}{m-2r}\binom{r}{i-j-r}\binom{m-2r}{r}\hspace{2.7cm}(\text{by Lemma \ref{cyclelemma}, (iii)})\\
&=\frac{2m}{m-2r}\binom{m-2r}{r}\binom{m+r}{i-r}\hspace{7cm}(\text{by (\ref{secondidentity})}).
\end{align*}
Assuming the formula is true for $l\geq 2$, we calculate it for $l+1$. Using Proposition \ref{joinoftwographs} again we get
\begin{align*}
&\beta_{i,i+r}(\underbrace{C_m*\dots *C_m}_{l+1\text{-times}})\\
&=\sum_{j=0}^{i+r-2}\left[\binom{m}{j}\beta_{i-j,i-j+r}(\underbrace{C_m*\dots *C_m}_{l\text{-times}}))+\binom{lm}{j}\beta_{i-j,i-j+r}(C_m))\right] \\
&=\sum_{j=0}^{i+r-2}\binom{m}{j}\frac{lm}{m-2r}\binom{m-2r}{r}\binom{(l-1)m+r}{i-j-r}+\sum_{j=0}^{i+r-2}\binom{lm}{j}\frac{m}{m-2r}\binom{r}{i-j-r}\binom{m-2r}{r} \\
&=\frac{m}{m-2r}\binom{m-2r}{r}\left[l\sum_{j=0}^{i-r}\binom{m}{j}\binom{(l-1)m+r}{i-r-j}+\sum_{j=0}^{i-r}\binom{lm}{j}\binom{r}{i-j-r}\right]\\
&=\frac{(l+1)m}{m-2r}\binom{m-2r}{r}\binom{lm+r}{i-r}\hspace{10cm}(\text{by (\ref{secondidentity})}).
\end{align*}
\qed \\ 
\end{proof}
Finally, we would like to calculate $\beta_{i,i+r}(G)$, where $r=\mathrm{reg}(R/I(G))$. There are three cases depending on $m$ modulo $3$.
\begin{theorem}\label{regularityBettinumber}
Let $m\geq 5$ be an integer and $r=\mathrm{reg}(R/I(G))=\mathrm{reg}(R/I(C_m))$. Then for $l\geq 2$, 
\[
\beta_{i,i+r}(G)=\begin{cases}
           3l\binom{(l-1)m+r}{i-r}-l\binom{(l-1)m}{i-m+r}\hspace{0.6cm}\text{for } m\equiv0 \pmod 3,\\
           lm\binom{(l-1)m+r}{i-r}+l\binom{(l-1)m}{i-m+r}\quad\text{for }m\equiv1 \pmod 3,\\
           l\binom{(l-1)m}{i-m+r}\hspace{3.1cm}\text{for }m\equiv2 \pmod 3.
           \end{cases}.
\]
\end{theorem}

\begin{proof}
We prove this by induction on $l$. Let $p=\mathrm{pd}(C_m)$. Assume that $m\equiv 0 \pmod 3$, i.e., $m=3k$ for some $k\geq 2$.
First we check the formula for $l=2$. We have by Proposition \ref{joinoftwographs},

\begin{align*}
  \beta_{i,i+r}(C_m*C_m)&=2\sum_{j=0}^{i+r-2}\binom{m}{j}\beta_{i-j,i-j+r}(C_m)\\
  &=2\sum_{j=0}^{i+r-2}3\binom{m}{j}\binom{r}{i-j-r}-2\binom{m}{i-p}\hspace{3cm}\text{(by Lemma \ref{cyclelemma} (iv) (a))}\\
  &=6\binom{m+r}{i-r}-2\binom{m}{i-m+r}\hspace{5.7cm}\text{(by Lemma \ref{cyclelemma} (i))}.
\end{align*}

We now verify the formula for $l+1$ assuming it is true for $l\geq 2$.

\begin{align*}
  &\beta_{i,i+r}(\underbrace{C_m*\dots *C_m}_{l+1\text{-times}})\\
  &=\sum_{j=0}^{i+r-2}\binom{m}{j}\beta_{i-j,i-j+r}(\underbrace{C_m*\dots *C_m}_{l\text{-times}})+\sum_{j=0}^{i+r-2}\binom{lm}{j}\beta_{i-j,i-j+r}(C_m)\hspace{1.2cm}\text{(by Proposition \ref{joinoftwographs})}\\
  &=\sum_{j=0}^{i-r}\binom{m}{j}3l\binom{(l-1)m+r}{i-j-r}-\sum_{j=0}^{i-p}\binom{m}{j}l\binom{(l-1)m}{i-j-p}+\sum_{j=0}^{i-r}3\binom{lm}{j}\binom{r}{i-j-r}\\
  &\hspace{10cm}-\binom{lm}{i-p}\quad(\text{by Lemma \ref{cyclelemma}, (iv)(a)}).\\
 \end{align*}
The induction is completed by using Equation \ref{secondidentity} and the fact that $p+r=m$ (see Lemma \ref{cyclelemma} (i)).

 The formulas for the cases $m\equiv 1 \pmod 3$ and $m\equiv 2 \pmod 3$ can be verified similarly using Lemma \ref{cyclelemma} (iv) (b) and (c), respectively.
\qed \\
\end{proof}

\begin{remark}
The formulas for $\beta_{i,i+j}(C_m*C_m)$ is obtained in \cite{MM}, which is $l=2$ case in Lemma \ref{joingeneral}. We have included the proof for $l=2$ case to make it self-contained. Also our formula in Theorem \ref{regularityBettinumber} for $l=2$ is slightly different than theirs. The calculations done here are inspired by those in \cite{MM}. 
\end{remark}

\begin{theorem}\label{cyclecombinatorialproperties}
  Let $l\geq 1$ and $m\geq 3$ be integers. Suppose $G=C_{lm}(1,\dots ,\widehat{2l},\dots ,\widehat{3l},\dots ,\lfloor{\frac{lm}{2}\rfloor})$. Then
 \begin{enumerate}[(i)]
 \item the induced matching number $\nu(G)=\lfloor\frac{m}{3}\rfloor$.
  \item $G$ is well-covered/Buchsbaum if and only if $m\in \{3,4,5,7\}$.
  \item $G$ is vertex decomposable/shellable/Cohen-Macaulay/sequentially Cohen-Macaulay if and only if either $l=1$ and $m\in\{3,5\}$ or $l\geq 2$ and $m=3$.
 \item $G$ is $S_2$ if and only if either $l=1$ and $m\in\{3,5,7\}$ or $l\geq 2$ and $m=3$.
 \end{enumerate}
\end{theorem}

\begin{proof}
 The circulant graph $G=\underbrace{C_m*\dots *C_m}_{l\text{-times}}$, by Lemma \ref{joingeneral}. If $l=1$, then $\nu(G)=\nu(C_m)=\lfloor\frac{m}{3}\rfloor$. When $l\geq 2$, by Proposition \ref{joinproperties}, $\nu(G)=\nu(C_m)=\lfloor\frac{m}{3}\rfloor$.
 For the remaining cases we may subdivide the proof into two cases: $l=1$ and $l\geq 2$. Statements for $l=1$ can be deduced from Proposition \ref{cycleproperties} and the $l\geq 2$ case is obtained by applying Proposition \ref{joinproperties}.
\qed \\ 
 \end{proof}








\section{{\bf Circulant graphs  $C_{lm}(1,\dots ,\widehat{l},\dots ,\widehat{2l},\dots ,\lfloor\frac{lm}{2}\rfloor)$}}\label{multi}

In this section $G$ denotes the circulant graph $C_n(1,\dots ,\widehat{l},\dots ,\widehat{2l},\dots ,\widehat{3l},\dots ,\lfloor{\frac{n}{2}\rfloor})$, where $n=lm$ is any composite number for $l,m \geq 2$.
\begin{lemma}\label{multipartite}
 Let $l\geq 2$ and $m_1,\dots ,m_l\geq 2$ be integers. The complete multipartite graph $K_{m_1,\dots ,m_l}$ is a circulant graph if and only if $m_i=m$ for $1\leq i\leq l$. In addition, if this is the case, then $K_{\underbrace{m,m,\dots ,m}_{l\text{-times}}}=C_{lm}(1,\dots ,\widehat{l},\dots ,\widehat{2l},\dots ,\lfloor\frac{lm}{2}\rfloor)$. 
\end{lemma}

 \begin{proof}
 If $m_i\neq m_j$ for some $i\neq j$, then $K_{m_1,\dots ,m_l}$ is not a regular graph and hence cannot be a circulant graph. To prove the second part, we partition the set $V(G)$ into $l$-components:
 $V_i=\{i, l+i, \dots, (m-1)l+i\}$, for $0\leq i < l$. Note that, the induced subgraphs $G_{V_i}$ consists of only isolated vertices. Also, for all $i,j$ with $i\neq j$ there is an edge between each vertex of $V_i$ and $V_j$, thus proving the statement.\qed \\
\end{proof}

 See, for example, the graph $C_{12}(1,\widehat{2},3,\widehat{4},5,\widehat{6}) \cong K_{6,6}$ in Figure \ref{figure 4} 
 \begin{figure}[ht!]
\centering
\begin{tikzpicture}
[scale=.55]
\draw [fill] (0,0) circle [radius=0.1];
 \draw [fill] (1.5,0) circle [radius=0.1];
\draw [fill] (2.4,0.8) circle [radius=0.1];
\draw [fill] (2.75,1.8) circle [radius=0.1];
\draw [fill] (2.7,2.9) circle [radius=0.1];
\draw [fill] (2.35,3.8) circle [radius=0.1];
\draw [fill] (1.5,4.6) circle [radius=0.1];
\draw [fill] (0,4.6) circle [radius=0.1];
\draw [fill] (-0.85,3.8) circle [radius=0.1];
\draw [fill] (-1.25,2.9) circle [radius=0.1];
\draw [fill] (-1.25,1.8) circle [radius=0.1];
\draw [fill] (-0.85,0.8) circle [radius=0.1];
\node at (0,-0.6) {$0$};
\node at (1.5,-0.6) {$1$};
\node at (2.9,0.5) {$2$};
\node at (3.25,1.7) {$3$};
\node at (3.2,2.9) {$4$};
\node at (2.87,3.89) {$5$};
\node at (1.9,4.9) {$6$};
\node at (-0.3,4.90) {$7$};
\node at (-1.37,3.9) {$8$};
\node at (-1.75,2.9) {$9$};
\node at (-1.95,1.7) {$10$};
\node at (-1.45,0.5) {$11$};
\draw (2.75,1.8)--(-1.25,1.8)--(2.35,3.8);
\draw (0,4.6)--(0,0);
\draw (0,0)--(1.5,0)--(2.4,0.8)--(2.75,1.8)--(2.7,2.9)--(2.35,3.8)--(1.5,4.6)--(0,4.6)--(-0.85,3.8)--(-1.25,2.9)--(-1.25,1.8)--(-0.85,0.8)--(0,0)--(2.75,1.8)--(1.5,4.6)--(-1.25,2.9)--(0,0);
\draw (0,0)--(2.35,3.8)--(-0.85,3.8)--(-0.85,0.8)--(2.4,0.8)--(2.35,3.8);
\draw (2.75,1.8)--(-0.85,3.8)--(1.5,0)--(2.7,2.9)--(0,4.6)--(-1.25,1.8)--(1.5,0)--(1.5,4.6);
\draw (1.5,4.6)--(-0.85,0.8);
\draw (-0.85,0.8)--(2.7,2.9)--(-1.25,2.9);
\draw (-1.25,2.9)--(2.4,0.8)--(0,4.6);
\node at (5.5,2.5) {$\cong$};
\draw [fill] (9,0) circle [radius=0.1];
\draw [fill] (9,1) circle [radius=0.1];
\draw [fill] (9,2) circle [radius=0.1];
\draw [fill] (9,3) circle [radius=0.1];
\draw [fill] (9,4) circle [radius=0.1];
\draw [fill] (9,5) circle [radius=0.1];
\draw [fill] (18,0) circle [radius=0.1];
\draw [fill] (18,1) circle [radius=0.1];
\draw [fill] (18,2) circle [radius=0.1];
\draw [fill] (18,3) circle [radius=0.1];
\draw [fill] (18,4) circle [radius=0.1];
\draw [fill] (18,5) circle [radius=0.1];
\draw (9,0)--(18,0)--(9,1)--(18,1)--(9,2)--(18,2)--(9,3)--(18,3)--(9,4)--(18,4)--(9,5)--(18,5)--(9,4)--(18,2)--(9,1)--(18,3)--(9,5)--(18,2)--(9,0)--(18,3)--(9,2)--(18,0)--(9,3)--(18,1)--(9,0)--(18,4)--(9,1)--(18,5)--(9,3)--(18,4)--(9,2)--(18,5)--(9,0);
\draw (18,0)--(9,4)--(18,1)--(9,5)--(18,0);
\node at (8.3,0) {$0$};
\node at (8.3,1) {$2$};
\node at (8.3,2) {$4$};
\node at (8.3,3) {$6$};
\node at (8.3,4) {$8$};
\node at (8.3,5) {$10$};
\node at (18.7,0) {$1$};
\node at (18.7,1) {$3$};
\node at (18.7,2) {$5$};
\node at (18.7,3) {$7$};
\node at (18.7,4) {$9$};
\node at (18.7,5) {$11$};
\end{tikzpicture}
\caption{$C_{12}(1,\widehat 2,3,\widehat 4,5,\widehat 6) \cong K_{6,6}$}\label{figure 4}
\end{figure}
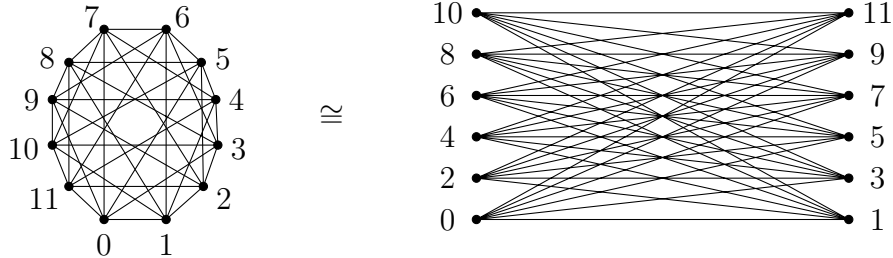

\begin{theorem}\label{Bettimulti}

The $\mathbb N$-graded Betti numbers of $R/I(G)$ are independent of the characteristic of field $\mathbb{K}$ and may be written as 
\[\beta_{i,i+1}(G)=\sum_{r=2}^l(r-1)\sum_{\underset{1\leq\alpha_1,\dots,\alpha_r\leq m}{\alpha_1+\dots +\alpha_r=i+1,}}\binom{l}{r}\binom{m}{\alpha_1}\dots\binom{m}{\alpha_r},\]
and $\beta_{i,d}(G)=0$ for $d\neq i+1$.

\end{theorem}
\begin{proof}
The proof follows from Lemma \ref{multipartite} and Theorem \ref{Jac}.
\qed \\
\end{proof}

\begin{remark}
 As we can write $G=\underbrace{G'*\dots *G'}_{l\text{-times}}$, where $G'$ is a graph consisting of $m$ number of isolated vertices, the result in Theorem \ref{Bettimulti} can also be deduced by applying Proposition \ref{joinoftwographs}.
 \end{remark}

The following proposition is a direct application of Proposition \ref{joinmultipartite}.

\begin{theorem}\label{multipartiteproperties}
For the circulant graph $G=C_n(1,\dots ,\widehat{l},\dots ,\widehat{2l},\dots ,\widehat{3l},\dots ,\lfloor{\frac{n}{2}\rfloor})$, 

\begin{enumerate}[(i)]
\item $\mathrm{reg}(R/I(G))=1$, $\mathrm{pd}(R/I(G))=n-1$ and $\nu(G)=1$.
\item $G$ is well-covered and Buchsbaum. 
\item $G$ does not satisfy any of the following properties: vertex decomposability/shellability/Cohen-Macaulayness/sequentially Cohen-Macaulayness or Serre's condition $S_2$.
\end{enumerate}
\end{theorem}
\noindent
\textbf{Acknowledgements}

\noindent
 We are grateful to Dr Chanchal Kumar who established the contact between the two authors and also for his support and encouragements. The second author is thankful to CSIR, Government of India for financial support.

\end{document}